\documentclass[a4paper, 11pt]{amsart}

\usepackage{amssymb}
\usepackage{amsmath}
\usepackage{amsthm}
\usepackage{amsfonts}
\usepackage{enumerate}
\usepackage{tikz}
\usepackage[left=3cm,right=3cm]{geometry}

\theoremstyle{plain}
\newtheorem{theorem}{Theorem}[section]
\newtheorem{lemma}[theorem]{Lemma}
\newtheorem{prop}[theorem]{Proposition}

\theoremstyle{definition}
\newtheorem{definition}[theorem]{Definition}

\theoremstyle{remark}

\newcommand{\mb}[1]{\mathbb{#1}}

\newcommand{\mfr}[1]{\mathfrak{#1}}

\title{On improving a Schur-type theorem in shifted primes}

\author{Ruoyi Wang}
\address{Edinburgh, Scotland}
\email{zoe.wang.maths@gmail.com }
\thanks{The work was conducted during the author's DPhil study, which was supported by a Clarendon Scholarship of the University of Oxford and a Jason Hu Scholarship of Balliol College.}

\begin{document}

\begin{abstract}
We show that if $N \geq \exp(\exp(\exp (k^{O(1)})))$, then any $k$-colouring of the primes that are less than $N$ contains a monochromatic solution to $p_1 - p_2 = p_3 -1$.
\end{abstract}

\maketitle

\section{Introduction}
Schur~\cite{Sch17} proved that the equation $x+y=z$ is partition regular over the set of the natural numbers in his celebrated paper published over a hundred years ago. More specifically, let $r(k)$ be the smallest positive integer such that every $k$-colouring of the set $[r(k)]$ contains a monochromatic solution to $x + y = z$ where $x,y,z \in [r(k)]$. It follows from Schur's paper that one has 
\begin{equation}\label{Schur integers}
\exp(\Omega (k)) \leq r(k) \leq \exp (O(k \log k)).
\end{equation}

There is an extensive collection of work on improving bounds for $r_p(k)$, see Irving~\cite{Irv73} for bounding it from above, see Heule~\cite{Heu18} and its reference concerning the lower bound.

Li and Pan~\cite{LP12} showed that the equation $x+y=z$ is also partition regular over the sparse set $\mb{P} -1$, where $\mb{P}$ denotes the primes. Let $r_p (k)$ be the smallest integer such that for any $N \geq r_p (k)$, any $k$-colouring of $\mb{P} \cap [N]$ contains a monochromatic solution to
\begin{equation*} 
p_1 - p_2 = p_3 - 1.
\end{equation*}
Although not explicitly stated, the following bound follows from the proof of Li and Pan. 
\begin{theorem}[Li and Pan~\cite{LP12}]\label{Li--Pan bound}
We have $$r_p (k) \leq \exp(\exp(\exp(\exp(k^{O(1)})))).$$
\end{theorem}

We prove the following quantitative strengthening.
\begin{theorem}\label{improved PR bound}
We have
$$r_p (k) \leq \exp(\exp(\exp(k^{O(1)}))).$$
\end{theorem}

Li and Pan used a somewhat similar strategy as Green's~\cite{Gre05b} proof of Roth's theorem in the primes which can be summarised as follows. Firstly, one needs the counting result in the setting of the integers, which are essentially quantitative versions of Schur's theorem and Roth's theorem, respectively. Secondly, one aims to prove the corresponding counting result in the primes. This is done by the transference principle, which typically involves two parts ---
an analytic framework which transfers the result in the setting of the integers to suitably weighted functions, and finding an appropriate way to weight the integers in a way that resembles the primes\footnote{Precisely speaking, the weight used in the work of Green and Li--Pan resembles a subset of primes which is contained in an arithmetic progression. One needs to choose the common difference of the progression carefully to avoid the Fourier bias caused by small primes.}.

We shall take a rather different approach by using the method from~\cite{San20} and applying existing proofs~\cite{RS08, Wan20} concerning the upper bound on the cardinality of subsets of $[N]$ whose difference sets avoid $\mb{P}-1$. These proofs~\cite{RS08, Wan20} follow from a Fourier concentration argument which differs from the $L^\infty$-transference principle used in the proof of Li and Pan. There are three major steps involved in the proof of Theorem~\ref{improved PR bound}. Firstly, we derive an iteration lemma which helps us to find many shifted primes in the difference set of any set $A$ with large density. By the pigeonhole principle, there exists a large monochromatic set amongst those shifted primes, and it turns out that a translate of this monochromatic set has a large intersection with the original set $A$. These considerations ultimately yield a colouring bootstrapping lemma. The final part of the proof is to bootstrap this lemma. We eventually conclude that there is always a monochromatic solution to $p_1 - p_2 = p_3 -1$, where $p_1, p_2, p_3 \in [N] \cap \mb{P}$, unless $N$ is small in terms of $k$, which completes the proof. 

Since the major arc estimate used to prove Lemma~\ref{iteration} holds only for $\mb{P} \pm 1$, our method cannot be used to study the partition regularity of $p_1 - p_2 = p_3 - t$, where $p_1, p_2, p_3 \in \mb{P}$, when $t \neq \pm 1$. The easiest example to compare the problems is when $t = p$ where $p$ is a prime number. In this case, the equation $p_1 - p_2 = p_3 - p$ always has the trivial monochromatic solution $p_1=p_2=p_3=p$. Nevertheless, as there always exists an arithmetic progression whose difference set avoids $\mb{P} - p$, we obviously cannot prove a corresponding difference set result in this case.

There has been a noticeable interest in understanding partition regularity problems over sparse sets. L\^e~\cite{Le12} proved a more general partition regularity result for linear equations over shifted primes, showing that any system of linear equations which is partition regular over $\mb{N}$ is also partition regular over $\mb{P}-1$. The author uses ideas from the paper of Green and Tao~\cite{GT08} concerning more general linear equations in the primes instead of Green's work on 3-term arithmetic progressions, together with Deuber’s work~\cite{Deu73} on partition regularity. On the other hand, a famous problem of Erd\H{o}s and Graham~\cite{Gra07} asking whether the equation $x + y = z$ is partition regular over perfect squares is still open.

The note is organised as follows. Section 2 is a list of notation. We summarise necessary estimates for the Fourier transform in Section 3, which are from the existing literature. In Section 4, we modify an iteration lemma of Ruzsa and Sanders and apply the lemma iteratively to locate shifted primes. The key bootstrapping lemma is shown in Section 5. In the final section, we apply the bootstrapping lemma from Section 5 to prove Theorem~\ref{improved PR bound}.

\section{Notation}
The set of primes is denoted by $\mb{P}$, and $\Lambda$ denotes the von Mangoldt function. For any positive integers $N,d$, let
\begin{equation}\label{def}
F_{N,d} (n) =  \Lambda(dn + 1) 1_{[N]} (n).
\end{equation}
Let $\chi$ be a Dirichlet character of modulus $q$. We use $L(s, \chi)$ to denote the associated Dirichlet $L$-function. Let
$$\psi(x; q,a) = \sum_{\substack{n \leq x \\ n \equiv a \pmod{q}}} \Lambda(n).$$

Let $f \in \ell^1 (\mb{Z})$. The Fourier transform of $f$ is defined as the function $\widehat{f} : \mb{T} \to \mb{C}$ given by
$$\widehat{f} (\theta) = \sum_{x \in \mb{Z}} f(x) e(-x\theta),$$
where $e(\theta) = e^{2 \pi i \theta}.$ The convolution of two functions $f,g \in \ell^1 (\mb{Z})$ is defined by
$$f * g (x) = \sum_{y \in \mb{Z}} f(x-y) g(y).$$

\section{Fourier transform of $F_{N,d}$ }
In this section, we summarise and simplify several existing results from the literature~\cite{RS08, Wan20}.

\subsection{Preliminaries}

Using the classical zero-free region~\cite[Theorem 5.26]{IK04} and a result of Landau~\cite[Theorem 5.28]{IK04}, one can show the following lemma.
\begin{lemma}\label{dichotomy lemma}
There exists a positive constant $c$ such that the following holds for any $D \geq 2$. If there exists a primitive character $\chi_D$ such that $\chi_D$ has modulus $d_D \leq D$ and $L(s,\chi_D)$ has a zero $\beta_D$ in the region
\begin{equation}\label{lem 3.1 region}
\Re(s) \geq 1 - \dfrac{c}{\log(D(|\Im(s)| + 3))},
\end{equation}
then 
\begin{enumerate}[(i)]
\item the zero $\beta_D$ is real and simple, and it is the only zero of $L(s,\chi_D)$ in the region (\ref{lem 3.1 region});
\item there does not exist any other primitive character $\chi$ of modulus\footnote{The absolute constant $10$ is chosen for the purpose of proving Lemma~\ref{iteration}.} $q \leq D^{10}$ such that $L(s,\chi)$ has a zero in the region (\ref{lem 3.1 region}).
\end{enumerate}

\end{lemma}
As a consequence of the lemma above, the following definition is exhaustive.
\begin{definition}\label{dichotomy}
Let $c$ be as in Lemma~\ref{dichotomy lemma} and $D \geq 2$. We say that $D$ is {\it exceptional} if there exists a unique primitive character $\chi_D$ such that $\chi_D$ has modulus $d_D \leq D$, and $L(s,\chi_D)$ has a zero $\beta_D$ which is real and simple and satisfies $\beta_D \geq 1 - c/ \log (3D)$. We call $\chi_D$ the {\it exceptional character} and $\beta_D$ the {\it exceptional zero}. Otherwise, we say that $D$ is {\it unexceptional}.
\end{definition}

If $\chi \pmod{q}$ has an exceptional zero $\beta$, then it follows from~\cite[Theorem 5.28]{IK04} that
\begin{equation}\label{class number formula}
1 - \beta \gg q^{-1/2},
\end{equation}
where the implicit constant is effective. We can use (\ref{class number formula}) and a prime number theorem for arithmetic progressions to prove the following estimate for the Fourier transform of $F_{N,d}$ at $0$.
\begin{prop}\label{FT at zero}
Let $D \geq 2$ and let $N, d$ be positive integers. Suppose that either
\begin{enumerate}
\item $D$ is unexceptional, $\bar{d}=1$ and $d \leq D$,
\item or $D$ is exceptional, $\bar{d} = d_D$ and $d \leq D^9.$
\end{enumerate}
Then
$$|\widehat{F_{N,\bar{d}d}} (0)| \gg \dfrac{ dN}{ \phi(\bar{d}d)} + O\left( \bar{d}dN \exp \left( - c\dfrac{\log (\bar{d}dN)}{\sqrt{\log(\bar{d}dN)} + \log(\bar{d}d)} \right) (\log (\bar{d} d ))^4 \right).$$
Here $c$ is a positive absolute constant which does not depend on $N,D,d$.
\end{prop}
\begin{proof}[Sketch of the proof]
The unexceptional case is a direct consequence of the definition of Fourier transform and a well-known prime number theorem for arithmetic progressions~\cite[Theorem 5.27]{IK04}. When $D$ is exceptional, we have 
\begin{multline*}
 \widehat{F_{N,\bar{d}d}} (0) = \psi(\bar{d}dN + 1; \bar{d}d, 1) \\
=  \dfrac{\bar{d}dN}{\phi(\bar{d}d)} - \dfrac{(\bar{d}dN)^{\beta_D}}{\phi(\bar{d}d)\beta_D} + O\left( \bar{d}dN \exp \left( - c\dfrac{\log (\bar{d}dN)}{\sqrt{\log(\bar{d}dN)} + \log(\bar{d}d)} \right) (\log (\bar{d} d))^4 \right).
\end{multline*}
Using the inequality $1 - (\bar{d} d N)^{\beta_D - 1}/\beta_D \geq 1-\beta_D$ which holds\footnote{This can be verified by taking derivatives of both sides and noticing the upper bound (\ref{class number formula}) for $\beta_D$.} for $N \geq 100$ and $\beta_D \geq 1/2$, we conclude that 
\begin{equation}\label{FT at 0 proof}
 \widehat{F_{N,\bar{d}d}} (0) \gg \dfrac{\bar{d} d N}{\phi(\bar{d}d)} ( 1- \beta_D) + O\left( \bar{d}dN \exp \left( - c\dfrac{\log (\bar{d}dN)}{\sqrt{\log(\bar{d}dN)} + \log(\bar{d}d)} \right) (\log (\bar{d} d))^4 \right).
\end{equation}
Combining the inequality above with (\ref{class number formula}) yields the proposition.
\end{proof}

\subsection{The major arcs}

The following major arc estimate is used in the proof of the iteration lemma. 
\begin{prop}\label{major arcs prop}
Let $D \geq 2$. Let $N, d, a, q$ be positive integers such that $(a,q) =1$. Suppose that either
\begin{enumerate}
\item $D$ is unexceptional, $\bar{d}=1$ and $dq \leq D$,
\item or $D$ is exceptional, $\bar{d} = d_D$ and $dq \leq D^9.$
\end{enumerate}
Then for any $\delta \in [-1/2, 1/2]$, we have
\begin{multline*}
\left| \widehat{F_{N,\bar{d}d}} \left( \dfrac{a}{q} + \delta \right) \right| \ll \dfrac{ |\widehat{F_{N,\bar{d}d}} (0)| }{\phi(q)} + O \left( (1 + |\delta| N) \bar{d}dq N \exp\left( - c \dfrac{ \log (\bar{d}dq^2N)}{\sqrt{\log N} + \log D} \right) \log^4 (\bar{d}dq) \right).
\end{multline*}
Here $c$ is a positive absolute constant which does not depend on $N,D,d,a,q,\delta$.
\end{prop}

\begin{proof}[Sketch of the proof]
We sketch a proof for the exceptional case. The unexceptional case can be shown in a similar manner. By a change of variables, one can show that
\begin{align*}
\widehat{F_{N,\bar{d}d}} \left(\dfrac{a}{q} + \delta \right) = & \sum_{m = 0}^{q-1} e \left( - \dfrac{am}{q} \right) \sum_{\substack{l \leq \bar{d}dN+1 \\ l \equiv 1 \pmod{\bar{d}d} \\ \frac{l-1}{\bar{d}d} \equiv m \pmod{q}}}  \Lambda (l)  e^{ -2\pi i \delta (l-1)/(\bar{d}d)}.
\end{align*}
It then follows from Lemma~\ref{dichotomy lemma}, a prime number theorem for arithmetic progressions~\cite[Theorem 5.27]{IK04} and standard analytic number theory techniques (such as integration by parts and partial summation) that
\begin{align*}
&\widehat{F_{N,\bar{d}d}} \left(\dfrac{a}{q} + \delta \right) \\
 = & \dfrac{1}{\phi(\bar{d}dq)} \sum_{m=0}^{q-1} e \left(-\dfrac{am}{q} \right) \int_{1}^{\bar{d}dN+1} e^{-2\pi i \delta t / (\bar{d} d)}\left( \overline{\chi_0 (\bar{d}dm+1)}- \overline{\chi_0\chi_D (\bar{d}dm+1)} t^{\beta_D-1} \right) dt \\
 + & O \left( (1+ | \delta |N ) \bar{d} dq^2 N \exp\left( - c \dfrac{ \log (\bar{d}dqN)}{\sqrt{\log N} + \log D} \right) \log^4 (\bar{d}dq)\right),
\end{align*}
where $\chi_0$ is the principal character of modulus $\bar{d}dq$. Since $\chi_0\chi_D(\bar{d} k + 1) =\chi_D(1) = 1$ for every nonnegative integer $k$ such that $(\bar{d} k + 1, \bar{d}dq)=1$, and $\chi_0\chi_D(\bar{d} k + 1) =0$ otherwise, it follows that 
\begin{align*}
& \sum_{m=0}^{q-1} e \left(-\dfrac{am}{q} \right) \int_{1}^{\bar{d}dN+1} e^{-2\pi i \delta t / (\bar{d} d)}\left( \overline{\chi_0(\bar{d}dm+1)} - \overline{\chi_0 \chi_D (\bar{d}dm+1)} t^{\beta_D-1} \right) dt  \\
= & \int_{1}^{\bar{d}dN+1} e^{-2\pi i \delta t / (\bar{d} d)}\left( 1- t^{\beta_D-1} \right) dt \sum_{\substack{m=0 \\ (\bar{d}dm+1,q)=1}}^{q-1} e \left(-\dfrac{am}{q} \right).
\end{align*}
Notice that the inner sum over $m$ is $0$ if $(\bar{d}d,q) >1$, and it follows from the well-known bound for Ramanujan sum (see Iwaniec--Kowalski~\cite[(3.3)]{IK04}) that its absolute value is bounded by $1$ if $(\bar{d}d,q) = 1$.

Since $1 - t^{\beta_D - 1} \geq 1 - \beta_D$ for all $t \gg 1$, it follows from H\"older's inequality that
\begin{align*}
&\left| \widehat{F_{N,\bar{d}d}} \left(\dfrac{a}{q} + \delta \right) \right| \leq  \dfrac{1}{\phi(\bar{d}d)\phi(q)} (1 - \beta_D) \bar{d} d N \\
 + & O \left( (1+ | \delta |N ) \bar{d} dq^2 N \exp\left( - c \dfrac{ \log (\bar{d}dqN)}{\sqrt{\log N} + \log D} \right) \log^4 (\bar{d}dq)\right).
\end{align*}
Comparing the inequality above with (\ref{FT at 0 proof}) yields the result.
\end{proof}

\subsection{The minor arcs} The minor arc estimate we need is a rather well-known consequence of the minor arc estimate of Vinogradov (see, for instance, Iwaniec and Kowalski~\cite[Theorem 13.6]{IK04}).
\begin{prop}\label{minor arc estimate}
Let $N, Q, d$ be positive integers such that $d \leq N$. For any positive integers $q,a$ such that $q \leq Q$ and $(a,q)=1$ and any $ \left| \theta - {a}/{q}\right| \leq {1}/({qQ})$, we have
$$| \widehat{F_{N,d}} (\theta)| \ll d (\log N)^4 \left( \dfrac{N}{\sqrt{q}} + N^{4/5} + \sqrt{NQ} \right).$$
\end{prop}

\section{Locating shifted primes in difference sets}
This section consists of two lemmas. The first one is an iteration lemma used in the difference set problem. Then we iterate this result to deduce the second one.

The set up of the lemma differs depending on whether a possible exceptional zero occurs or not. In the exceptional case, the iteration starts from an arithmetic progression whose common difference is a multiple of the exceptional modulus. This is motivated by considering the model situation where the Fourier transform is concentrated near the arcs whose denominators are multiples of the exceptional modulus.

We need some notation to describe the circle method. Let $Q$ be a positive integer. For any integers $q \leq Q$ and $1 \leq a \leq q$, we define
$$\mfr{M}_{a,q} = \left\{ \theta \in \mb{T}: \left| \theta - \dfrac{a}{q} \right| \leq \dfrac{1}{qQ} \right\}$$
and
$$\mfr{M}^*_q = \bigcup_{(a,q) = 1, 1 \leq a \leq q} \mfr{M}_{a,q}.$$
By Dirichlet's pigeonhole principle, we have
$$\mb{T} = \bigcup_{q \leq Q} \mfr{M}^*_q.$$

The following result of Ruzsa and Sanders~\cite[Corollary 7.3]{RS08} is used to locate density increment. It is purposely designed according to the major arc estimate for $\widehat{F_{N,d}}$.
\begin{prop}\label{Cor 7.3 RS}
Let $N$ be a positive integer and $A \subseteq [N]$ have density $\alpha >0$. Let $Q \geq Q_1 \geq 1$. Suppose
$$\alpha^{-1} |A|^{-1} \sum_{q = 1}^{Q_1} \dfrac{1}{\phi(q)}  \int_{\mfr{M}_q^* } |(1_A - \alpha 1_{[N]} ) \;\widehat{} \;(\theta) |^2 d \theta \geq c$$
for some $c>0$, then there exists an arithmetic progression $P$ with common difference $q \leq Q_1$ and $|P| \gg Q_1^{-1} \min \{ Q, c \alpha N\}$ such that $|A \cap P| \geq \alpha (1 + 2^{-5} c) |P|$.
\end{prop}

The proof of the lemma below is essentially the same as the proof of Ruzsa and Sanders~\cite{RS08}. Instead of looking at whether the difference set contains any shifted primes, we quantify the number of shifted primes that can be located using the same idea.

\begin{lemma}\label{iteration}
There exist positive absolute constants $C_1, c_1$ such that we can obtain the following result.

Let $D \geq 2$ and let $N$ be a positive integer such that\footnote{We need to introduce an upper bound on $N$ due to the factor $(\log N)^4$ in the minor arc estimate.} $N \leq \exp(D^{1/10})$. Let $d$ be a positive integer. Let $\bar{d} = 1$ if $D$ is unexceptional, or $\bar{d} = d_D$ if $D$ is exceptional. 

Suppose $A \subseteq [N]$ has density $\alpha >0$. Suppose also that
\begin{equation}\label{1}
C_1 (\log D)^2 \leq \log N \text{ and }
\log d + \log \alpha^{-1} \leq C_1^{-1} \log D.
\end{equation}
Then one of the following assertions must be true.
\begin{enumerate}[(i)]
\item There exists an arithmetic progression $P'$ with common difference $\ll \alpha^{-3}$ and length $\gg (\alpha / \bar{d}d \log N)^8 N$ such that $|A \cap P'| \geq \alpha (1 + c_1) |P'|.$
\item There exists $N' \gg \alpha N$ such that
\begin{equation*}
\left| (A - A) \cap \left( \dfrac{\mb{P} - 1}{\bar{d} d} \right) \right| \geq \dfrac{ c_1 \alpha  N' }{ \bar{d} \log N'}.
\end{equation*}

\end{enumerate}
\end{lemma}

\begin{proof}

Let $c>0$ be an absolute constant to be optimised later. Then, either $N \ll c^{-1}\alpha^{-1}$ which is impossible for $C_1 \gg_c 1$ due to (\ref{1}), or $N' := \lfloor c\alpha N \rfloor \geq 1$. By Proposition~\ref{FT at zero}, we have 
\begin{equation}\label{lower bound at 0}
|\widehat{F_{N',\bar{d}d}} (0) | \gg \dfrac{N'}{\bar{d}}.
\end{equation}

Suppose first that
\begin{equation}\label{counting weighted primes iteration}
\langle 1_A * 1_{-A}, F_{N' ,\bar{d}d} \rangle \geq \dfrac{ \alpha^2 N |\widehat{F_{N', \bar{d}d}}(0)|}{2}.
\end{equation}
Notice that there are at most $O((\bar{d}dN')^{1/2})$ composite prime powers which are less than $\bar{d}dN' +1$. Combine this observation with the trivial upper bounds $1_{A} * 1_{-A} (n) \leq \alpha N$ and $F_{N' ,\bar{d}d} (n) \ll \log(\bar{d}dN')$, we have
$$\sum_{\substack{n: \bar{d}dn+1 = p^k \\ p \text{ prime}, k \geq 2}} 1_A * 1_{-A} (n)  F_{N' ,\bar{d}d} (n) \ll \alpha N \sqrt{\bar{d}dN'} \log (\bar{d}dN').$$
Since $C_1$ is large, it follows from (\ref{1}) that $|\widehat{F_{N', \bar{d}d}}(0)|^{2/3}  \gg \alpha \sqrt{\bar{d}dN'} \log (\bar{d}dN')$, and so
$$\sum_{\substack{n: dn+1 = p^k \\ p \text{ prime}, k \geq 2}} 1_A * 1_{-A} (n)  F_{N' ,\bar{d}d} (n) \ll \alpha^2 N |\widehat{F_{N', \bar{d}d}}(0)|^{2/3}.$$
By (\ref{lower bound at 0}) and the definition of $N'$, if $|\widehat{F_{N', d}}(0)| = O(1)$ then (\ref{1}) can no longer hold. Thus, subtracting the inequality above from (\ref{counting weighted primes iteration}) yields
$$\sum_{\substack{n: \bar{d}dn+1 \in \mb{P}}} 1_A * 1_{-A} (n)  F_{N' ,\bar{d}d} (n) \geq \dfrac{\alpha^2 N |\widehat{F_{N', \bar{d}d}}(0)|}{4}.$$
Using the bounds $1_{A} * 1_{-A} (n) \leq \alpha N$ and $F_{N' ,\bar{d}d} (n) \leq \log(\bar{d}dN'+1)$ again, together with H\"older's inequality, (\ref{1}) and (\ref{lower bound at 0}), we can conclude that
\begin{equation}\label{outcome 2}
\left| (A - A) \cap \{n: \bar{d} d n + 1 \in \mb{P} \} \right| \gg \dfrac{  \alpha N' }{ \bar{d} \log N'}.
\end{equation}

Otherwise, we have 
\begin{equation}\label{counting weighted primes iteration 2}
\langle 1_A * 1_{-A}, F_{N' ,\bar{d}d} \rangle \leq \dfrac{ \alpha^2 N |\widehat{F_{N', \bar{d}d}}(0)|}{2}.
\end{equation}
Let $I$ denote the interval $[N]$ and consider the inner product
$$\langle ({1}_A - \alpha {1}_I) * ({1}_{-A} - \alpha {1}_{-I} ), F_{N',\bar{d}d} \rangle.$$
By (\ref{counting weighted primes iteration 2}) and direct computations, we have
$$\langle ({1}_A - \alpha {1}_I) * ({1}_{-A} - \alpha {1}_{-I} ), F_{N',\bar{d}d} \rangle \leq \alpha^2 N |\widehat{F_{N', \bar{d}d}}(0)| (-1/2 + O(c)).$$
Therefore, we can choose $c \gg 1$ which guarantees
$$|\langle ({1}_A - \alpha {1}_I) * ({1}_{-A} - \alpha {1}_{-I} ), F_{N',\bar{d}d} \rangle | \gg \alpha^2 N |\widehat{F_{N', \bar{d}d}}(0)|.$$
It follows from Plancherel's theorem that
\begin{equation}\label{energy}
\int_{\mb{T}} \left| \left(\widehat{1_A} - \alpha \widehat{1_I} \right) (\theta) \right|^2 |\widehat{F_{N', \bar{d}d}} (\theta) | d\theta \gg \alpha^2 N |\widehat{F_{N', \bar{d}d}}(0)|.
\end{equation}

Let $c', c''$ be two positive absolute constants to be chosen later. Let 
$$Q' = \dfrac{(\bar{d}d)^4 \log^8 N'}{c'^2 \alpha^2}, \; Q = \dfrac{N'}{Q'} \; \text{and}\; Q'' = c'' \alpha^{-3}.$$
If $Q' \geq Q''$ then let
$$\mfr{M}' := \bigcup_{Q' < q \leq Q} \mfr{M}^*_q, \; \mfr{M}'' := \bigcup_{Q'' < q \leq Q'} \mfr{M}^*_q  \text{  and  } \mfr{M}''' := \bigcup_{q \leq Q''} \mfr{M}^*_q,$$
which are defined at the beginning of this section. If $Q' < Q''$ we define $\mfr{M}'' = \emptyset$ and $\mfr{M}', \mfr{M}'''$ as above. Notice that it follows from Dirichlet's pigeonhole principle that $\mb{T} = \mfr{M}' \cup \mfr{M}'' \cup \mfr{M}'''$,
and so
\begin{multline}\label{decomposition of T v2}
\int_{\theta \in \mb{T}} | (\widehat{1_A} - \alpha \widehat{1_I} ) (\theta) |^2 |\widehat{F_{N', \bar{d}d}} (\theta)| d\theta
\leq \int_{\theta \in \mfr{M}'} | (\widehat{1_A} - \alpha \widehat{1_I} ) (\theta) |^2 |\widehat{F_{N', \bar{d}d}} (\theta)| d\theta 
\\ + \int_{\theta \in \mfr{M}''} | (\widehat{1_A} - \alpha \widehat{1_I} ) (\theta) |^2 |\widehat{F_{N', \bar{d}d}} (\theta)| d\theta +\int_{\theta \in \mfr{M}'''} | (\widehat{1_A} - \alpha \widehat{1_I} ) (\theta) |^2 |\widehat{F_{N', \bar{d}d}} (\theta)| d\theta.
\end{multline}

By (\ref{1}), we can apply the minor arc estimate given in Theorem~\ref{minor arc estimate} with our choices\footnote{We also used $N \leq \exp(D^{1/10})$ in the inequality below.} of $N', Q'$ and $Q$ to see that
\begin{equation}\label{upper bound - minor arcs}
|\widehat{F_{N', \bar{d}d}} (\theta) | \ll \dfrac{N'}{\bar{d}d} \left( c' \alpha + \bar{d}d N'^{-1/10} \right) \ll \dfrac{c' \alpha N'}{\bar{d}d} \ll c' \alpha |\widehat{F_{N',\bar{d}d}} (0)|
\end{equation}
for all $\theta \in \mfr{M}'$. 
By Plancherel's theorem,
\begin{equation*}
\int_{\mb{T}} \left|  (\widehat{1_A} - \alpha \widehat{1_I} ) (\theta) \right|^2 d\theta \leq \alpha N.
\end{equation*}
Hence, it follows from H\"older's inequality and (\ref{upper bound - minor arcs}) that
\begin{equation}\label{energy - large range}
\int_{\theta \in \mfr{M}'} \left|  (\widehat{1_A} - \alpha \widehat{1_I} ) (\theta) \right|^2 |\widehat{F_{N', \bar{d}d}} (\theta) | d\theta \ll c' \alpha^2 N  |\widehat{F_{N',\bar{d}d}}(0)|.
\end{equation}

We employ our major arc estimate on $\mfr{M}''$ and $\mfr{M}'''$. Since for all $C_1 \gg_{c'} 1$, we have $dQ' \leq D^9$ if $D$ is exceptional or $dQ' \leq D$ if $D$ is unexceptional, it follows from Proposition~\ref{major arcs prop} that for $a/q + \delta \in \mfr{M}_{a,q}$ we have
\begin{multline}\label{major arc - 2}
 \left| \widehat{F_{N',\bar{d}d}} \left( \dfrac{a}{q} + \delta \right) \right|  \ll  \dfrac{ |\widehat{F_{N',\bar{d}d}} (0)| }{\phi(q)} \\
 + O \left( (1 + |\delta| N') \bar{d}dq N' \exp\left( - \Omega \left( \dfrac{ \log (\bar{d}dqN')}{\sqrt{\log N'} + \log D} \right) \right) \log^4 (\bar{d}dq) \right).
\end{multline}
Notice that for any $C_1$ which is sufficiently large in terms of $c'$ and the implicit constants in (\ref{lower bound at 0}) and (\ref{major arc - 2}), it follows from our choices of $Q, Q'$ and (\ref{1}) that for each $q \leq Q'$, the second term in the major arc estimate (\ref{major arc - 2}) is at most $|\widehat{F_{N', \bar{d}d}}(0)|/\phi(q)$. Therefore, uniformly for all $q \leq Q'$ and $a/q + \delta \in \mfr{M}_{a,q}$ we have
\begin{equation}\label{major arc range}
\left| \widehat{F_{N',\bar{d}d}} \left( \dfrac{a}{q} + \delta \right) \right| \ll \dfrac{|\widehat{F_{N', \bar{d}d}}(0)|}{\phi(q)}.
\end{equation}
It follows from the choice of $Q''$ and the definition of $\mfr{M}''$ that
\begin{equation}\label{sub-critical range}
\sup_{\theta \in \mfr{M}''} |\widehat{F_{N', \bar{d}d}} (\theta) |  \ll c'' \alpha{|\widehat{F_{N', \bar{d}d}} (0)|}.
\end{equation}
By (\ref{sub-critical range}) and Plancherel's theorem, we have
\begin{equation}\label{energy - sub-critical range}
\int_{\theta \in \mfr{M}''} |  (\widehat{1_A} - \alpha \widehat{1_I} ) (\theta) |^2 |\widehat{F_{N', \bar{d}d}} (\theta)| d\theta \ll c'' \alpha^2 N |\widehat{F_{N', \bar{d}d}} (0) |.
\end{equation}

Therefore, combining (\ref{energy}), (\ref{energy - large range}), (\ref{energy - sub-critical range}) and (\ref{decomposition of T v2}), we can choose $c' \gg 1$ and $c'' \gg 1$ so that
$$\int_{\theta \in \mfr{M}'''} |  (\widehat{1_A} - \alpha \widehat{1_I} ) (\theta) |^2 |\widehat{F_{N', \bar{d}d}} (\theta)| d\theta \gg \alpha^2 N |\widehat{F_{N', \bar{d}d}} (0) |.$$
By the definition of $\mfr{M}'''$ and the triangle inequality, we have
$$\sum_{q \leq Q''} \int_{\theta \in \mfr{M}^*_q} |  (\widehat{1_A} - \alpha \widehat{1_I} ) (\theta) |^2 |\widehat{F_{N', \bar{d}d}} (\theta)| d\theta \gg  \alpha^2 N |\widehat{F_{N', \bar{d}d}} (0) |.$$
Combine this with (\ref{major arc range}), we can deduce that
$$\sum_{q \leq Q''} \dfrac{|\widehat{F_{N', \bar{d}d}} (0) | }{\phi(q)} \int_{\theta \in \mfr{M}^*_q} |  (\widehat{1_A} - \alpha \widehat{1_I} ) (\theta) |^2 d\theta \gg \alpha^2 N |\widehat{F_{N', \bar{d}d}} (0) |.$$
Since $|\widehat{F_{N', \bar{d}d}} (0) |>0$, we obtain the density increment outcome (i) by applying Proposition~\ref{Cor 7.3 RS}.

\end{proof}

Now we apply iteratively the lemma above to find shifted primes in the difference set. Assuming the input parameters satisfy certain conditions, which roughly say that $A$ has large density and the input common difference is not too large, we rule out density increment and locate many shifted primes in an arithmetic progression.

\begin{lemma}\label{inner iteration PR}
Let $N$ and $d$ be positive integers. Let $D \geq 2$ and $\alpha>0$. Suppose that $N < \exp (D^{1/10})$ and $\alpha = O ( (\log N)^{-1})$. Suppose also that
\begin{equation}\label{PR inner iteration - size assumption}
C_{2}  (\log D)^2 \leq \log N \text{ and } \log d + (\log (\alpha^{-1}))^2 \leq C_2^{-1} \log D
\end{equation}
for some sufficiently large constant $C_2$. Let $\bar{d} = d_D$ be the modulus of the exceptional zero if $D$ is exceptional, or $\bar{d} = 1$ if $D$ is unexceptional.

Suppose that $A$ is contained in an arithmetic progression of length $N$ with common difference $\bar{d}d$, and has $|A| = \alpha N$. Then, there exist $A' \subseteq A - A$ which is contained in an arithmetic progression of length $N' \geq \alpha^{C_3 (\log \alpha^{-1})^2}  (\bar{d}d)^{ - C_3 \log\alpha^{-1}} (\log N)^{- C_3 \log\alpha^{-1}} N$ with common difference $\bar{d} d'$, where $d' \leq \alpha^{-C_3 \log\alpha^{-1}} d$ and $d \mid d'$, such that
$$|A' \cap (\mb{P} - 1)|  \geq  \dfrac{c_1 \alpha N'} {\bar{d} \log N'}.$$
Here, $C_3$ is an absolute constant which does not depend on $N, \alpha, d, \bar{d}$ or $D$, and $c_1$ is the same one as in Lemma~\ref{iteration}. 
\end{lemma}

\begin{proof}
Let $A_0$ be the affine transformation of $A$ so that $A_0 \subseteq [N]$. We choose 
$$A_1 = A_0, N_1 = N, \alpha_1 = \alpha \text{ and } d_1 = d.$$ 
Since $C_2$ is large, $D, \bar{d}, A_1, N_1, \alpha_1, d_1$ satisfy the hypotheses of Lemma~\ref{iteration}.

Suppose for some positive integer $k$, we have obtained $A_k$, $N_k$, $\alpha_k$ and $d_k$ such that $A_k$ is a subset of $[N_k]$, $|A_k| = \alpha_k N_k$, $d \mid d_k$, $A_k$ is dilated by a factor $\bar{d}^{-1} d_k^{-1}$ times a shifted subset of $A$, the inputs $D, \bar{d}, A_k, N_k, \alpha_k, d_k$ satisfy the hypotheses of Lemma~\ref{iteration}, and $N_k$, $\alpha_k$, $d_k$ satisfy the bounds
\begin{equation}\label{density bound PR chap}
\alpha_{k} \geq (1 + c_1 )^{k-1} \alpha,
\end{equation}
\begin{equation}\label{size d PR chap}
d_{k} \leq \alpha^{-C(k-1)} d,
\end{equation}
and
\begin{equation}\label{size N PR chap}
N_{k} \geq \dfrac{ \alpha^{Ck^2} N} { (\bar{d} d)^{C(k-1)} (\log N)^{C(k-1)}},
\end{equation}
where $C$ is a large absolute constant which is independent of any parameters. Let us apply Lemma~\ref{iteration} to $D$, $\bar{d}$, $A_k$, $N_k$, $\alpha_k$ and $d_k$. Either we are in outcome (ii), or it follows from outcome (i) that there exist $A_{k+1}$, $N_{k+1}$, $\alpha_{k+1}$ and $d_{k+1}$ such that $A_{k+1}$ is a subset of $[N_{k+1}]$, $|A_{k+1}| = \alpha_{k+1} N_{k+1}$, $d \mid d_{k+1}$, $A_{k+1}$ is dilated by a factor $\bar{d}^{-1} d_{k+1}^{-1}$ times a shifted subset of $A$, the inputs $N_{k+1}$, $\alpha_{k+1}$, $d_{k+1}$ satisfy the bounds $\alpha_{k+1} \geq (1 + c_1 )^{k} \alpha,$ $d_{k+1} \leq \alpha^{-Ck} d,$ and $N_{k+1} \geq \alpha^{C(k+1)^2} N { (\bar{d} d)^{-Ck} (\log N)^{-Ck}}$. Moreover, for sufficiently large $C_2$, condition (\ref{1}) from Lemma~\ref{iteration} holds for all $k \leq 2(\log \alpha^{-1}) / \log (1+ c_1)$, where $c_1$ is the constant from Lemma~\ref{iteration}. 

Since outcome (i) at every step of the iteration would produce a set with density bigger than $1$ when $k > (\log  \alpha^{-1} )/ \log (1+ c_1)$, it follows that we must be in outcome (ii) for some 
\begin{equation}\label{size steps PR chap}
k_0 \leq\dfrac{\log \alpha^{-1}}{\log (1+ c_1)}.
\end{equation}
Using the size bounds (\ref{density bound PR chap}), (\ref{size d PR chap}) and (\ref{size N PR chap}) and the bound given in outcome (ii), we have
\begin{equation}\label{shifted primes outcome iteration}
\left| (A_{k_0} - A_{k_0}) \cap \left( \dfrac{\mb{P} - 1}{\bar{d}d_{k_0}} \right) \right| \geq \dfrac{ c_1 \alpha_{k_0}  N'_{k_0}}{ \bar{d} \log N'_{k_0}},
\end{equation}
where
\begin{equation}\label{size N' PR chap}
N_{k_0}' \gg \alpha_{k_0} N_{k_0} \gg \dfrac{ \alpha^{Ck_0^2+2} N} { (\bar{d}d)^{Ck_0} (\log N)^{Ck_0}}.
\end{equation}
Let 
$$
A'_{k_0} =  \{ n \in {A_{k_0}} - {A_{k_0}}: n \leq N'_{k_0}, \bar{d} d_{k_0} n + 1 \text{ is a prime} \}.$$
It follows from (\ref{shifted primes outcome iteration}) that
\begin{equation}\label{size A' PR chap}
|A'_{k_0}|  \geq \dfrac{ c_1 \alpha_{k_0}  N'_{k_0}}{ \bar{d} \log N'_{k_0}}.
\end{equation}
Let $A'$ be the pre-image of $A'_{k_0}$ under the affine transformations carried out in the proof above. Since $A_{k_0}$ is dilated by a factor $\bar{d} d_{k_0}^{-1}$ times a shifted subset of $A$, it follows that $A'$ is contained in an arithmetic progression of length $N_{k_0}'$ with common difference $\bar{d} d_{k_0}$, and $A'$ consists of numbers which are one less than a prime. Therefore, one can conclude that the assertion of the lemma holds for sufficiently large $C_3$ by (\ref{size d PR chap}), (\ref{size steps PR chap}), (\ref{size N' PR chap}) and (\ref{size A' PR chap}).
\end{proof}

\section{The bootstrapping lemma}
The aim of this section is to establish the bootstrapping lemma which is the key tool used in the proof of Theorem~\ref{improved PR bound}.

We need the following preliminary lemma which follows from averaging.

\begin{lemma}\label{PR averaging}
Let $X \subseteq \mb{Z}$ be a set contained in an arithmetic progression of length $N$ with common difference $d$. Suppose $Y \subseteq \mb{Z}$ is contained in an arithmetic progression of length $N'$ with common difference $dd'$ where $d'$ is a positive integer, then there exists $n \in \mb{Z}$ such that
$$| (n +Y) \cap X | \geq \dfrac{|X||Y|}{N + d' N'}.$$
\end{lemma}
\begin{proof}
We assume $d = 1$ in the proof below, since the general result follows from applying the $d=1$ result to affine transformations of the sets $X$ and $Y$.

Notice that the set
$$\{ x \in \mb{Z} : (x + Y) \cap X \neq \emptyset \}$$
is contained in an interval of length at most $N + d'N'$. Also, by a direct counting argument, we have
$$\sum_{x \in \mb{Z}} | (x + Y) \cap X| = |X| |Y|.$$
Therefore, it follows from the pigeonhole principle that there exists some $n \in \mb{Z}$ such that
\begin{align*}
| (n + Y) \cap X| \geq \dfrac{|X||Y|}{N +d'N'},
\end{align*}
which completes the proof.
\end{proof}

We combine Lemma~\ref{inner iteration PR} and the averaging lemma above to deduce the bootstrapping lemma. There are two steps involved in the proof of the bootstrapping lemma below. The first one is to apply the iteration lemma to an initial set and the pigeonhole principle to find either a desired monochromatic solution, or a large monochromatic subset of shifted primes in its difference set. In the latter case, we then use the averaging lemma above to locate one translate whose intersection with the initial set is large. Since difference sets are translation invariant, the difference set of this particular translate lies in the difference set of the initial set, and so it does not contain the colour of the initial set.

\begin{lemma}\label{PR iteration}
Let $N_0$ and $k$ be positive integers. Let $N \leq N_0$ and $d$ be positive integers. Let $D \geq 2$ be a positive number such that $N < \exp(D^{1/10})$ and let $\alpha = O( (\log N)^{-1} )$. Let $\bar{d} = d_D$ be the modulus of the exceptional zero if $D$ is exceptional, or $\bar{d} = 1$ if $D$ is unexceptional.
 
Suppose $N,d,\alpha, D$ satisfy (\ref{PR inner iteration - size assumption}), then for any $k$-colouring of $(\mb{P} -1) \cap [N_0]$ and $j \leq k$, we have the following.

Suppose $A \subseteq [N_0]$ is a monochromatic subset of $\mb{P}-1$ contained in an arithmetic progression of length $N$ with common difference $d\bar{d}$, and has $|A| = \alpha N$. Suppose also that $A \cup ((A-A) \cap (\mb{P}-1)) $ is $j$-coloured, then 
\begin{enumerate}
\item either there exists a monochromatic solution to $x - y = z$ where $x, y, z \in (\mb{P}-1) \cap [N_0]$;
\item or there exists a monochromatic set $A'' \subseteq \mb{P}-1$ satisfying the following:
 \begin{enumerate}
 \item $A''$ is contained in an arithmetic progression of length $N''$ and common difference $d''\bar{d}$, where  $N'' \geq \alpha^{C (\log \alpha^{-1})^2} (dD)^{ - C \log \alpha^{-1}} (\log N)^{ - C \log \alpha^{-1}} N$, $d \mid d''$ and $d'' \leq \alpha^{-C\log\alpha^{-1}} d$;
 \item $|A''| \geq {\alpha^{C}} ({(j-1) \bar{d} \log N})^{-1} N''$; 
 \item $A'' \cup ((A'' - A'') \cap (\mb{P}-1))$ is $(j-1)$-coloured.
 \end{enumerate}
\end{enumerate}
\end{lemma}
\begin{proof}

Suppose there exists $z \in (A-A) \cap (\mb{P} -1)$ which shares the same colour as $A$, then we are in the first outcome. Hence, from now on, we assume that $(A-A) \cap (\mb{P} -1)$ is at most $(j-1)$-coloured, and every element of $(A-A) \cap (\mb{P}-1)$ has a different colour from $A$. 

By applying Lemma~\ref{inner iteration PR} with $A, \alpha, d, D, N$, we can find a set $A' \subseteq A - A$ such that $A'$ is contained in an arithmetic progression of length 
$$N' \geq \dfrac{ \alpha^{C_3 (\log \alpha^{-1})^2} N}{(\bar{d}d)^{C_3 \log\alpha^{-1}} (\log N)^{C_3 \log\alpha^{-1}}}$$ 
with common difference $d'\bar{d}$ where
$$d' \leq  \alpha^{-C_3\log\alpha^{-1}} d,  \;\; d \mid d',$$ 
and
$$|A'  \cap (\mb{P} - 1)  | \geq \dfrac{c_1 \alpha}{\bar{d} \log N'} N'.$$
Here $c_1$ and $C_3$ are the same constants as the ones in Lemma~\ref{inner iteration PR}.

Let $B$ be the largest colour class in $A'\cap (\mb{P}-1)$. It follows that $B$ is contained in an arithmetic progression of length $N'$ with common difference $d'\bar{d}$, and
$$|B| \geq \dfrac{c_1 \alpha}{ ( j - 1) \bar{d} \log N'} N'.$$
By Lemma~\ref{PR averaging}, there exists some $n$ such that
\begin{align*}
|(n + B) \cap A| \geq & \dfrac{|A| |B|}{N + (d'/d)N'}.
\end{align*}
Since $A' \subseteq A-A$, we have $d'N' \leq 2dN$, and so
\begin{align*} 
|(n + B) \cap A|  \geq & \dfrac{|A||B|}{3N} \\
\geq & \dfrac{ c_1 \alpha^2 }{3 (j-1) \bar{d} \log N } N'.
\end{align*}
Let 
$$A'' := B \cap (A - n).$$
Since $A'' \subseteq B$, it follows that $A''$ is monochromatic and satisfies condition (a), with $N'' = N'$ and $d'' = d'$. We have
$$|A''| = |(n + B) \cap A| \geq \dfrac{ c_1 \alpha^2 }{3 (j-1) \bar{d} \log N } N''$$
which completes the proof of assertion (b). By our assumption which rules out outcome (i), we know that $A'' \subseteq B$ is monochromatic and has a different colour from $A$. Since $A'' \subseteq A - n$, we have $A'' - A'' \subseteq A - A$, and so $(A'' - A'') \cap (\mb{P}-1)$ contains at most $j-1$ colour classes and they are different from the colour of $A$. By combining these observations, we can conclude that assertion (c) is true.

\end{proof}

\section{Proof of Theorem~\ref{improved PR bound}}
The following lemma is used to modify the initial set if an exceptional zero occurs.
\begin{lemma}\label{modification 2}
Let $A$ be a subset of an arithmetic progression of length $N$ with common difference $d$ and have $|A| = \alpha N$ for some $\alpha > 0$. For every positive integer $\bar{d}$, there exists an arithmetic progression $P$ with common difference $d\bar{d}$ and length at least $\alpha N/\bar{d}$ such that $|A \cap P| \geq \alpha|P|/2$.
\end{lemma}
\begin{proof}
The lemma follows from applying Lemma~\ref{PR averaging} with $d' = \bar{d}$ and $Y = \{d d'  n: 0 \leq n \leq \alpha N/\bar{d}, n \in \mb{Z}\}$.
\end{proof}

The proof of the main theorem essentially follows from bootstrapping Lemma~\ref{PR iteration} from the previous section. The initial monochromatic set follows from the prime number theorem and a basic application of the pigeonhole principle. Afterwards, we shall repeatedly apply Lemma~\ref{PR iteration} until we find a monochromatic solution to $x + y = z$ in $\mb{P}-1$, or we shall run out of the colour classes which is impossible. In the proof below, we shall clarify our choices of sets at each step of the bootstrapping, and keep a detailed record of the quantities involved there.

\begin{proof}[Proof of Theorem~\ref{improved PR bound}]

Let $N_0 = N$. By the prime number theorem and the pigeonhole principle, the largest colour class in $(\mb{P}-1) \cap [N_0]$ must have cardinality at least $(1 + o(1)) N_0 (k \log N_0)^{-1}$. Thus, we can always find a monochromatic subset $A_0 \subseteq [N_0]$ such that $|A_0| = \alpha_0 N_0$ with $\alpha_0 = (2 k \log N_0)^{-1}$. 

We shall proceed by contradiction. Suppose the theorem is false, then we claim that for each $1 \leq i \leq k$, we can find $D_i$, $A_i$, $N_i$, $\alpha_i$, $\overline{d_i}$, $A''_i, N''_i, d''_i, \alpha''_i$ which satisfy the following conditions:
 \begin{enumerate}[(I)]
 \item $2^{-k-4+i} (\log (2\alpha_{i})^{-1})^3  \leq \log D_{i} \leq 2^{-k-4+i} (\log \alpha_{i}^{-1})^3$;
 \item $2^{-k-4+i} (\log (2\alpha_{i-1})^{-1} )^3 \leq \log \alpha_{i}^{-1} \leq 2^{-k-3+i} (\log \alpha_{i-1}^{-1})^3$;
 \item $\log d_i \leq 2^{-k-4+i} (\log \alpha_{i-1}^{-1})^3$;
 \item $\log N_i \geq \log N_{i-1} - C' (\log \alpha^{-1}_{i-1})^9$ where $C'$ is a large absolute constant;
 \item $A_{i}''$ is contained in an arithmetic progression of length $N_{i}''$ with common difference $d_{i}''\overline{d_{i}}$, where $N_{i}'' \geq \alpha_{i}^{C (\log \alpha_{i}^{-1})^2} (d_{i}D_{i})^{ - C \log \alpha_{i}^{-1}} (\log N_i)^{ - C \log \alpha_{i}^{-1}} N_{i}$, $d_{i} \mid d_{i}''$ and $d_{i}'' \leq \alpha_{i}^{-C\log\alpha_{i}^{-1}} d_{i}$;
 \item $|A_{i}''| = \alpha''_{i} N''_{i}$ where $\alpha''_i/2 \leq {\alpha_{i}^{C \log \alpha_{i}^{-1}}} ({(k-i) d_{i} D_{i} \log N_{i}})^{-1} \leq \alpha''_{i} $;
and
 \item $A_{i}'' \cup ((A_{i}'' - A_{i}'') \cap (\mb{P}-1))$ is $(k-i)$-coloured;
\end{enumerate}
In particular, since condition (VI) and condition (VII) are impossible for $i=k$, we come to a contradiction at the $k$th step and so the theorem must be true.

Let $D_1 = \exp( (\log \alpha_0^{-1})^3/2^{k+3})$ and $d_1 = 1$. 
\begin{itemize}
\item if $D_1$ is unexceptional, we let $N_1 = N_0$, $\overline{d_1} = 1$, and $A_1$ be a subset of $A_0$ such that $|A_1| = \alpha_1 N_1$ where\footnote{This is always possible since $\alpha_1$ stated here is way smaller than $\alpha_0$ since $k \leq (\log\log\log N_0)^c$ for some $c<1$. Similarly, we can make such a choice in the exceptional case.} $2^{-k-3} (\log (2\alpha_{0})^{-1} )^3 \leq \log \alpha_{1}^{-1} \leq 2^{-k-2} (\log \alpha_{0}^{-1})^3$;
\item if $D_1$ is exceptional, then let $\overline{d_1}$ be the exceptional modulus. We choose $A_1$ to be a subset of the set obtained from Lemma~\ref{modification 2}, so that $A_1$ is contained in an arithmetic progression with common difference $\overline{d_1}$ and length $N_1 \gg \alpha_0  N_0 /\overline{d_1}$, and $A_1$ has density $\alpha_1$ in this arithmetic progression where $2^{-k-3} (\log (2\alpha_{0})^{-1} )^3 \leq \log \alpha_{1}^{-1} \leq 2^{-k-2} (\log \alpha_{0}^{-1})^3$.
\end{itemize} 
By our choice of $D_1$ and the bound $k \leq (\log\log\log N_0)^c$ for some small constant $c$, the parameters satisfy the conditions listed in (\ref{PR inner iteration - size assumption}), and so we can apply Lemma~\ref{PR iteration} to conclude that either the theorem is true (when the first outcome of Lemma~\ref{PR iteration} occurs) which contradicts our assumption, or we can find a monochromatic set $A''_1 \subseteq \mb{P}-1$ such that
 \begin{itemize}
 \item $A_1''$ is contained in an arithmetic progression of length $N_1''$ and common difference $d_1''\bar{d_1}$, where $N_1'' \geq \alpha_1^{C (\log \alpha_1^{-1})^2} (d_1D_1)^{ - C \log \alpha_1^{-1}} (\log N)^{ - C \log \alpha_1^{-1}} N_1$, $d \mid d''$ and $d_1'' \leq \alpha_1^{-C\log\alpha_1^{-1}} d_1$;
 \item $|A_1''| = \alpha''_1 N''_1$ where $\alpha''_1/2 \leq  {\alpha_1^{C \log \alpha_1^{-1}}} ({(k-1) d_1 D_1 \log N_1})^{-1} \leq \alpha''_1$;
 \item $A_1'' \cup ((A_1'' - A_1'') \cap (\mb{P}-1))$ is $(k-1)$-coloured.
\end{itemize}
In particular, it follows from the properties above that conditions (I)-(VII) hold for $i = 1$.

Suppose for some $1 \leq j \leq k-1$, we have defined $D_i$, $A_i$, $N_i$, $\alpha_i$, $\overline{d_i}$, $A''_i, N''_i, d''_i, \alpha''_i$ which satisfy conditions (I)-(VII) for all $1 \leq i \leq j$. We choose $D_{j+1}$, $A_{j+1}$, $N_{j+1}$, $\alpha_{j+1}$, $\overline{d_{j+1}}$, $A''_{j+1}, N''_{j+1}, d''_{j+1}, \alpha''_{j+1}$ as follows. Fix 
$$D_{j+1} = \exp( (\log {\alpha''_{j}}^{-1})^3/ 2^{k-j+3}) \;\;\text{and}\;\; d_{j +1} = d''_j \overline{d_j}.$$
Depending on whether $D_{j+1}$ is exceptional, we define $A_{j+1}, N_{j+1}, \overline{d_{j+1}}, \alpha_{j+1}$ as follows: 
\begin{enumerate}[({A}.I)]
\item if $D_{j+1}$ is unexceptional, we take $A_{j+1} = A''_j$, $N_{j+1} = N''_j$, $\overline{d_{j+1}} = 1$ and $\alpha_{j+1} = \alpha''_j$;
\item if $D_{j+1}$ is exceptional, then let $\overline{d_{j+1}}$ be the exceptional modulus. Let $A_{j+1}$ be the set obtained from Lemma~\ref{modification 2}, so that $A_{j+1}$ is contained in an arithmetic progression with common difference $d_{j+1}\overline{d_{j+1}}$ and length $N_{j+1} \gg \alpha''_j  N''_j /\overline{d_{j+1}}$, and $A_{j+1}$ has density $\alpha_{j+1} \geq \alpha''_j /2$ in this arithmetic progression.
\end{enumerate} 

We shall first show that conditions (I)-(IV) hold for $j+1$. 

Condition (I) is a direct consequence of the definitions of $D_{j+1}$ and $\alpha_{j+1}$.

By (A.I), (A.II) and (VI), we have
\begin{align*}
\log \alpha_{j+1} \geq & \log \alpha_j'' - \log 2 \\
\geq & - C (\log \alpha_j^{-1})^2 - \log (k-j) - \log d_j - \log D_j - \log\log N_j - \log 2.
\end{align*}
Since $k = (\log\log\log N_0)^c$ for some small constant $c<1$, and $\alpha_i \leq (k \log N_0)^{-1}$ for all $i \leq j$, we have $C (\log \alpha_j^{-1})^2 + \log (k-j) +  \log\log N_j + \log 2 \leq 2^{-k-5+j} (\log \alpha_j^{-1})^3$. By (II), (III) and $\alpha_j \leq (k\log N_0)^{-1}$, we also have 
\begin{align*}
\log d_j \leq 2^{-k-4+j} (\log \alpha_{j-1}^{-1})^3 \leq 2^{-k-1+j} \log \alpha_j^{-1} \leq 2^{-k-5+j} (\log \alpha_j^{-1})^3.
\end{align*}
Combine the bounds above with (I), we conclude that
\begin{align*}
\log \alpha_{j+1}^{-1} \leq  2^{-k-4+j} (\log \alpha_j^{-1})^3 + \log D_j \leq 2^{-k-3+j} (\log \alpha_j^{-1})^3.
\end{align*}
On the other hand, since $\log \alpha_{j+1}^{-1} \geq \log \alpha_j^{''-1}$, it follows from (I) and (VI) that 
$$\log \alpha_{j+1}^{-1} \geq  \log D_j \geq 2^{-k-4+j} (\log (2\alpha_j)^{-1})^3.$$  
Thus, we conclude that condition (II) holds for $j+1$.

To see that condition (III) is true, notice that combining our definitions of $d_{j+1}$ and $\overline{d_j}$, (I), (III) and (V) yields
\begin{align*}
\log d_{j+1} = & \log d''_j + \log \overline{d_j} \\
\leq & C (\log \alpha_{j}^{-1} )^2 + \log d_j + \log D_j \\
\leq & C (\log \alpha_{j}^{-1} )^2 + 2^{-k-4+j} (\log \alpha_{j-1}^{-1})^3 + 2^{-k -4+j}(\log \alpha_j^{-1})^3.
\end{align*}
Since $k \leq (\log\log\log N_0)^c$ and $\alpha_j \leq (k \log N_0)^{-1}$, it follows that one has $C(\log \alpha_{j}^{-1} )^2  \leq 2^{-k -5+j}(\log \alpha_j^{-1})^3$. On the other hand, by condition (II), $k \leq (\log\log\log N_0)^c$ and $\alpha_j \leq (k\log N_0)^{-1}$, we have $2^{-k-4+j} (\log \alpha_{j-1}^{-1})^3 \leq \log 2^{-k-1+j} \alpha_j^{-1} \leq  2^{-k-5+j} (\log \alpha_j^{-1})^3$. Therefore,
$$\log d_{j+1}  \leq 2^{-k-3+j} (\log \alpha_j^{-1})^3.$$

By (A.I), (A.II) and (V), one obtains
\begin{align*}
\log N_{j+1} \geq  & \log N''_j - \log D_{i+1} + \log \alpha''_j + O(1) \\
\geq & -C (\log \alpha^{-1}_j)^3 - C (\log d_j + \log D_j) \log \alpha_j^{-1}- C \log \alpha_j^{-1}\log\log N_j \\& + \log N_j  - \log D_{j+1} - C(\log \alpha_j^{-1})^2 \\
& - \log(k-j) -\log d_j - \log D_j -\log\log N_j + O(1).
\end{align*}
Notice that
$$ (\log d_j )(\log \alpha_j^{-1})+\log \alpha_j^{-1}\log\log N_j \ll  (\log \alpha^{-1}_j)^2 $$
which follows from (II), (III) and the inequality $\alpha_j \leq \alpha_0 \leq (k \log N)^{-1}$. Hence, by (I), (II), (A.I) and (A.II), one has
\begin{align*}
\log N_{j+1} \geq & -C (\log \alpha^{-1}_j)^3 - 2C (\log D_j) \log \alpha_j^{-1}+ \log N_j  - \log D_{j+1} + O ( (\log \alpha_j^{-1})^2) \\
\geq & \log N_j - C' (\log \alpha^{-1}_j)^9,
\end{align*}
and so (IV) holds for $j+1$.

In fact, it follows from the verifications above that the estimates from (\ref{PR inner iteration - size assumption}) hold for $A_{j+1}, D_{j+1}, \alpha_{j+1}, d_{j+1}$. The first inequality there follows from iterating (II) and (IV) and using (I). More specifically, one has
\begin{equation*}
\log \alpha^{-1}_{j+1} \leq  (\log \alpha^{-1}_0)^{3^{j+1}}
\end{equation*}
and
$$\log N_{j+1} \geq \log N_0 - (j+1) C'  (\log \alpha^{-1}_j)^9.$$
Therefore, the first inequality from (\ref{PR inner iteration - size assumption}) holds if
$$2^{-k-2+j}  (\log \alpha^{-1}_0)^{3^{j+3}} \ll  \log N_0 - (j+1) C'  (\log \alpha^{-1}_0)^{3^{j+3}}$$
for some sufficiently large implicit constant, which is true since $j \leq (\log\log\log N_0)^c$ where $c<1$ and $\alpha_0 \gg (k \log N_0)^{-1}$. 

By (III), we have
$$\log d_{j+1}  \leq 2^{-k-3+j} (\log \alpha_j^{-1})^3.$$
By (VI), we have $\log \alpha_j^{''-1} + \log 2 \geq C(\log \alpha_j^{-1})^2$, and so
$$\log \alpha_j^{-1} \leq (2 (\log \alpha_j^{''-1}) / C)^{1/2}.$$
Thus, it follows that
$$\log d_{j+1} \leq 2^{-k-3+j}(2/C)^{3/2} (\log \alpha_j^{''-1})^{3/2}.$$ 
Combining this inequality with the definition of $D_{j+1}$, $k \leq (\log\log\log N_0)^c$ and $\alpha_j \leq (k \log N_0)^{-1}$ yields
$$\log d_{j+1} \leq C_2^{-1} (\log D_{j+1}) / 2.$$
On the other hand, it follows from condition (I), $k \leq (\log\log\log N_0)^c$ and $\log \alpha_{j+1}^{-1} \gg \log(k\log N_0)$ that
$$(\log \alpha_{j+1}^{-1})^2  \leq C_2^{-1} (\log D_{j+1}) / 2.$$
Thus, the second inequality of (\ref{PR inner iteration - size assumption}) also holds.

Therefore, by Lemma~\ref{PR iteration}, either the theorem holds (if the first outcome of Lemma~\ref{PR iteration} occurs) which contradicts our assumption, or we can find a monochromatic set $A''_{j+1} \subseteq \mb{P}-1$ such that
 \begin{enumerate}[({B}.I)]
 \item $A_{j+1}''$ is contained in an arithmetic progression of length $N_{j+1}''$ and common difference $d_{j+1}''\overline{d_{j+1}}$, where $d_{j+1} \mid d_{j+1}''$, $d_{j+1}'' \leq \alpha_{j+1}^{-C\log\alpha_{j+1}^{-1}} d_{j+1}$, and $N_{j+1}'' \geq \alpha_{j+1}^{C (\log \alpha_{j+1}^{-1})^2} (d_{j+1}D_{j+1})^{ - C \log \alpha_{j+1}^{-1}} (\log N_{j+1})^{ - C \log \alpha_{j+1}^{-1}} N_{j+1}$;
 \item $|A_{j+1}''| = \alpha''_{j+1} N''_{j+1}$ where 
$$\alpha''_{j+1}/2 \leq {\alpha_{j+1}^{C \log \alpha_{j+1}^{-1}}} ({(k-j-1) d_{j+1} D_{j+1} \log N_{j+1}})^{-1} \leq \alpha''_{j+1};$$
 \item $A_{j+1}'' \cup ((A_{j+1}'' - A_{j+1}'') \cap (\mb{P}-1))$ is $(k-j-1)$-coloured.
\end{enumerate}
In particular, this implies conditions (V)-(VII) hold for $i = j+1$.

\end{proof}

\section{Concluding remarks}
It is interesting to understand the true size of $r_p(k)$. The possible exceptional zero has an impact on the size of $|\widehat{F_{N,d}}(0)|$, which ultimately leads to the bound in Theorem~\ref{improved PR bound}. It seems that we cannot rule out the possibility that almost all energy is concentrated on $\mfr{M}_{a,q}$ where $q$ is a multiple of the modulus of the exceptional zero $d_D$. A model case is to consider a set $A$ which has large density on $\{d_D n + b_i: n \leq N, 1 \leq i \leq j\}$, where $b_i$ are integers such that $\psi(d_D N; d_D, b_i) \geq d_D N/\phi(d_D)$ and $d_D$ is a product of small primes.

On the other hand, a weaker bound on the major arcs would lead to a slower density increment. 
One may compare Lemma~\ref{iteration} with the iteration of Pintz--Steiger--Szemer\'edi~\cite{PSS88} to see this.

\section*{Acknowledgement}
The author would like to thank Tom Sanders for introducing the problem and many suggestions, and Joni Ter\"av\"ainen for discussions related to the transference principle.

\end{document}